\documentclass[reqno]{amsart}

\begin{document}
\title[]
{Density property of certain sets and their applications}

\author[Sahoo]
{Manas R. Sahoo}

\address{Manas R. Sahoo \newline
School of Mathematical Sciences\\
National Institute of Science Education and Research\\
Bhimpur-Padanpur, Jatni, Khurda\\
Bhubaneswar-752050, India}
\email{\newline manas@niser.ac.in}

\thanks{}

\begin{abstract}
In this paper we show that certain sets are dense in $\mathbb{R}$. We give some applications. 
For example, we show an analytical proof that $q^{\frac{1}{n}}$, $q$ is a prime number and $e$; are irrational numbers. As another application we show:
If $f$ is an locally integrable function on $\mathbb{R}-\{0\}$ satisfying $\int_x ^{px}f(t)dt$ and $\int_x ^{qx}f(t)dt$ are constant with $\frac{\ln p}{\ln q}$ is an irrational
number; implies $f(t)=\frac{c}{t}\,\,\ a.e.$, where $c$ is constant; which is already considered in \cite{b1} for the case when $f$ is continuous.

\end{abstract}

\maketitle
\numberwithin{equation}{section}
\numberwithin{equation}{section}
\newtheorem{theorem}{Theorem}[section]
\newtheorem{remark}[theorem]{Remark}
\newtheorem{lemma}[theorem]{Lemma}
\newtheorem{definition}[theorem]{Defination}
\newtheorem{example}[theorem]{Example}
\newtheorem{corollary}[theorem]{Corollary}
\section{Introduction}
It is well known that the set $\mathbb{A}=\{m+nq : m, n \in \mathbb{Z}\}$ is a dense set in $\mathbb{R}$. Here we provide a proof using Engel expansion.
In \cite{b1}, the authors proved that a set of the form $\{\pm p^m q^n: m, n \in \mathbb{Z}\}$ is dense sub set of $\mathbb{R}$ iff $\frac{\ln p}{\ln q}$ is an irrational
number. We give here a different proof. Also the author proved the following fact: If $f$ is a continuous function on $\mathbb{R}-\{0\}$ satisfying $\int_x ^{px}f(t)dt$ and $\int_x ^{qx}f(t)dt$ are constant with $\frac{\ln p}{\ln q}$ is an irrational
number; implies $f(t)=\frac{c}{t}$, where $c$ is  constant. We extend this result to any integrble function. Another part of this paper is to give an equivalent characterization of 
irrational numbers. By using this or the variant of the proof of this equivalent statements,
 we show certain type of numbers are irrational. For example we show $e, q^\frac{1}{n}$($q$ is a prime number) are irrational numbers. For similar works, we cite \cite{A1, D,p}.

\section{Series Representation of Irrational number and density property}
For the completeness we give below the proof of Engel expansion.
\begin{theorem}
 For any irrational number $0<q<1$, there exist natural numbers $p_i\geq 2\,\,\ , i=1,2...$  with $p_i \leq p_{i+1}$ such that 
 \begin{equation}
  q=\displaystyle{ \sum_{i=1}^ \infty} \frac{1}{p_1.p_2....p_i}
  \label{1}
 \end{equation}
\end{theorem}
\begin{proof}
 Since $0< q< 1$, there exists natural number $p_1\geq 2$ such that $(p_1-1)q<1< p_1 q< 2$.
 Now Set $\alpha_0 =q,\,\, p_0 =2$ and define $\alpha_1 =p_1 q-1$, then $0<\alpha_1<1$. Choose integer $p_2$ such that  $$(p_2-1)\alpha_1 <1< p_2 \alpha_1< 2.$$ Define $\alpha_2 =p_2 \alpha_1-1$.
 Observe that $\alpha_1 \leq \alpha_0$, which implies $p_1 \leq p_2$.

 By induction, we construct  $p_n \in \mathbb{N}$ and $\alpha_n$ satisfying the property\\
 \begin{equation}
 \begin{aligned}
  &(p_{n+1}-1)\alpha_{n} <1< p_{n+1} \alpha_{n}< 2,\\
  &\alpha_{n}=p_{n} \alpha_{n-1} -1\,\, \textnormal{ and}\\
  &p_n \leq p_{n+1}.
  \label{2}
  \end{aligned}
  \end{equation}
 
 Assume we are given $p_{i}, \,\, 1 \leq i \leq  k+1$ and $\alpha_i$,  $1 \leq i \leq k$; satisfying 
 \begin{equation*}
 \begin{aligned}
 &(p_{i+1}-1)\alpha_{i} <1< p_{i+1} \alpha_{i}< 2 \\ &\alpha_{i}=p_{i} \alpha_{i-1} -1\,\,\, \textnormal{and}\,\,\, p_i \geq p_{i-1}
 \label{}
  \end{aligned}
  \end{equation*}
 we construct $p_{k+2}$ and $\alpha_{k+1}$ as follows:
 we take $\alpha_{k+1}=p_{k+1} \alpha_{k} -1$. Choose $p_{k+2}$ such that $(p_{k+2}-1)\alpha_{k+1} <1< p_{k+2} \alpha_{k}< 2$.
 Clearly $\alpha_{k+1}\leq \alpha_{k}$. This implies $p_{k+2} \geq p_{k+1}$.
 
 Equation \eqref{2} yields:
 \begin{equation*}
 \begin{aligned}
  &\frac{1}{p_{n}}\leq \alpha_{n-1} \leq \frac{2}{p_{n}} \\
  &\Rightarrow \Big|\alpha_{n-1}-\frac{1}{p_{n}}\Big| < \frac{1}{p_{n}}\\
  &\Rightarrow \Big|p_{n-1} \alpha_{n-2}-1 -\frac{1}{p_{n}}\Big| < \frac{1}{p_{n}}\\
  &\Rightarrow \Big|\alpha_{n-2}-\frac{1}{p_{n-1}} -\frac{1}{p_{n-1}p_{n}}\Big| < \frac{1}{p_{n-1}p_{n}}\\
  \end{aligned}
   \end{equation*}
 Continuing in this way by induction we get:
 
 $$\Big|q-\displaystyle {\sum_{i=1}^{n}}\frac{1}{p_1. p_2...p_i}\Big|< \frac{1}{p_1. p_2...p_n}.$$
 Passing to the limit as $n \rightarrow \infty$ equation \eqref{1} follows.
 \end{proof}
 
 \begin{theorem}
  Define $\mathbb{A}=\{m+nq : m, n \in \mathbb{Z}\}$, $q\in \mathbb{R}$. Then the following statements are equivalent.
  \begin{itemize}
   \item[1.] $q$ is an irrational number.
   \item[2.] There exist $z_n \in \mathbb{A}$, $n\in \mathbb{N}$ such that $z_n \rightarrow 0$.
   \item[3.] $\mathbb{A}$ is dense in $\mathbb{R}$.
  \end{itemize}
 \end{theorem}
 \begin{proof}
  $1 \Rightarrow 2$:  Let $q$ be an irrational number. With out loss of generality we can assume $0<q<1$, then by the above theorem,
  
  \begin{equation}
   \begin{aligned}
    &q=\displaystyle{ \sum_{i=1}^ \infty} \frac{1}{p_1.p_2....p_i}\\
    &\Rightarrow \Big|q-\displaystyle{ \sum_{i=1}^ n} \frac{1}{p_1.p_2....p_i}\Big| < \frac{2}{p_1.p_2....p_{n+1}}\\
    &\Rightarrow \Big|p_1.p_2....p_n\Big(q-\displaystyle{ \sum_{i=1}^ n} \frac{1}{p_1.p_2....p_i}\Big)\Big| < \frac{2}{p_{n+1}}\\
   \end{aligned}
\end{equation}
Also note that for $p_i < p_{i+1}$ holds infinitely often, othewise $q$ is a rational number. So $p_{n+1} \rightarrow \infty$ as $n\rightarrow \infty$.
Let $s_n= -\Big(\displaystyle{ \sum_{i=1}^ n} \frac{1}{p_1.p_2....p_i}\Big)p_1.p_2....p_n$,\,\, $r_n=p_1.p_2....p_n$, then $r_n$ and $s_n$ are integers. 
Also we have $z_n =r_n +q s_n \in \mathbb{A}$ tends to zero. 

$2 \Rightarrow 3$:
One can consider $z_n$s' are positive. Let $a, b \in \mathbb{R}$ and $a<b$. Since $\frac{1}{z_n}(b-a)$ tends to 
$\infty$, there exist $N_0$ such that $\frac{1}{z_{N_0}}a<t< \frac{1}{z_{N_0}}b$, for some integer $t$. 
This implies $a<t z_{N_0}<b$. Since $t z_{N_0} \in\mathbb{A}$,  $\mathbb{A}$ is dense in $\mathbb{R}$.

$3 \Rightarrow 1$: Let $\mathbb{A}$ is dense in $\mathbb{R}$, we have to show $q$ is irrational. By contary assume $q$ is a rational number of the form 
$q=\frac{m_0}{n_0}$, $m_0, n_0 \in \mathbb{Z}$ and $n_0 \neq 0$ . 
Clearly $n_0 \mathbb{A}\subset \mathbb{Z}$, so distance between two elements of $\mathbb{A}$ is atleast $\frac{1}{n_0}$. Hence $\mathbb{A}$ is not dense.
 \end{proof}
 
 \begin{corollary}
  $\mathbb{B}=\{\pm p^m q^n: m, n \in \mathbb{Z}\}$ is a dense subset of $\mathbb{R}$ iff $\frac{\ln p}{\ln q}$ is an irrational number.
 \end{corollary}
\begin{proof}
Consider the set $\bar{\mathbb{B}}=\{m\ln p+ n \ln q: m, n \in \mathbb{Z}\}$:\\
$$\{m\ln p+ n \ln q: m, n \in \mathbb{Z}\}= \ln q \{m \frac{\ln p}{\ln q}+ n: m, n \in \mathbb{Z}\}.$$
By theorem $(2.2)$, $\{m \frac{\ln p}{\ln q}+ n: m, n \in \mathbb{Z}\}$ is a dense subset of $\mathbb{R}$ iff
$\frac{\ln p}{\ln q}$ is an irrational number. Hence $\bar{\mathbb{B}}$ is a dense subset of $\mathbb{R}$.

Now we will show $\bar{\mathbb{B}}$ is dense in $\mathbb{R}$ iff ${\mathbb{B}}$ is dense in $\mathbb{R}$. Let $y> 0$.
there exist sequence $m_t\ln p+ n_t\ln q$ which converges to $\ln y $ as $t$ tends to $\infty$. Now 
\begin{equation*}
 |p^{m_t} q^{n_t}-y|=|exp(m_t\ln p+ n_t\ln q)-exp (\ln y)|=exp(c(t))|[(m_t\ln p+ n_t\ln q)- \ln y]|,
\end{equation*}
where $c(t)$ is a point lying between $(m_t\ln p+ n_t\ln q)$ and $\ln y$. Since $c(t)$ is bounded, $p^{m_t} q^{n_t}$ converges to $y$. So 
$\{p^m q^n: m, n \in \mathbb{Z}\}$ is dense subset of $[0,\infty)$. Hence $\mathbb{B}$ is dense in $\mathbb{R}$. Similarly one can show the converse.
\end{proof}

\section{Applications}
 \begin{example}
 If $q$ is a prime number, then for any $n\geq2 \in \mathbb{N}$, $q^\frac{1}{n}$ is an irrational number.
 \end{example}
\begin{proof}
Choose $m\in \mathbb{N}$ such that $m< q^\frac{1}{n} < m+1 \Rightarrow 0< q^\frac{1}{n}-m <1$. Now consider the set
$$\mathbb{A}=\Big\{ \displaystyle {\sum_{i=0}^{n-1}}c_i q^\frac{i}{n}: c_i\in \mathbb{Z}\Big\} $$.

For any $k\in \mathbb{Z}$, ${(q^\frac{1}{n}-m)}^k \in \mathbb{A} $. As $k \rightarrow \infty$,  $z_k={(q^\frac{1}{n}-m)}^k \in \mathbb{A} \rightarrow 0$. 
So for $a, b \in \mathbb{R}$, there exists $t\in \mathbb{Z}$ and $n_0\in \mathbb{N}$ such that $\frac{a}{z_{n_0}}< t<\frac{a}{z_{n_0}}$. 
This implies $a< tz_{n_0}< b$. As $tz_{n_0} \in \mathbb{Z}$, So $\mathbb{A}$ is dense in $\mathbb{R}$.

If $q^\frac{1}{n}$ is a rational number, then $q=\frac{r}{s}$. Cleary $s^n \mathbb{A} \subset \mathbb{Z}$. So the distance between any two umbers of $\mathbb{A}$
is atleast $\frac{1}{s^n}$, which is a contradiction. Hence $q^\frac{1}{n}$ is an irrational number.
\end{proof}

\begin{example}
 Any number $q$ of the form  \eqref{1}, with $p_i < p_{i+1}$ is an irrational number. In particular $e$ is an irrational number.
\end{example}
\begin{proof}
If $q$ is of the form  \eqref{1}, then  $p_1.p_2....p_n\Big(q-\displaystyle{ \sum_{i=1}^ n} \frac{1}{p_1.p_2....p_i}\Big)\rightarrow 0$. So there are elements 
$z_n \in \mathbb{A}$ tends to zero as $n$ tends to infinity. So by Theorem $(2.2)$ the result follows.
\end{proof}

\begin{example}
 If $f$ is a locally integrable function on $\mathbb{R}-\{0\}$ satisfying $\int_x ^{px}f(t)dt$ and $\int_x ^{qx}f(t)dt$ are constant with $\frac{\ln p}{\ln q}$ is an irrational
number; implies $f(t)=\frac{c}{t}$, where $c$ is  constant. 
\end{example}
\begin{proof}
Define the measure $\mu$ on multipicative group  $\mathbb{R^{+}}$ as follows: Let $E$ be any borel measurable set of $\mathbb{R^{+}}$.
Define $\mu(E)= \int_{E}f(y)dy$, then our claim is $\mu$ is a Haar measure on $\mathbb{R}$.
\begin{equation}
 \begin{aligned}
  \mu([a,b])&= \int_{a}^{b}f(y)dy\\
  \Rightarrow \mu([pa,pb])&=\int_{pa}^{pb}f(y)dy\\
  &=\int_{pa}^{a}f(y)dy+\int_{a}^{b}f(y)dy+\int_{b}^{pb}f(y)dy
  \end{aligned}
\end{equation}
 Since $\int_x ^{px}f(y)dy$ is constant. This implies $\int_{pa}^{a}f(y)dy+\int_{b}^{pb}f(y)dy$. So $\mu([pa,pb]=\int_{a}^{b}f(y)dy=\mu([a,b]$.
 Hense by approximation $\mu(pE)=\mu(E)\Rightarrow \mu(p^n E)=\mu(E)$. Following the same analysis we get, $\mu(p^n q^m E)=\mu(E)$. 
 Since the set $\{ p^m q^n: m, n \in \mathbb{Z}\}$ is dense subset of $\mathbb{R^{+}}$. So for any $a\in \mathbb{R^{+}}$, $\mu(aE)=\mu(E)$. This proves $\mu$ is a
 Haar measure.

 Note that $\bar{\mu}(E)=\int_{E}\frac{1}{t}$ is a Haar measure on the multipicative topological group $\mathbb{R^{+}}$. 
 Applying \cite[Chapter 9, Theorem(11.9)]{f1}, $\mu=c\bar{\mu}$, for some $c\in \mathbb{R}$. This in turn gives $f(t)=\frac{c_1}{t}$ on $\mathbb{R^{+}}$. 
 
 Similarly considering same Haar measure concept on $\mathbb{R^{+}}$, with $f(t)$ is replaced by $f(-t)$, we can discover $f(-t)=\frac{c_2}{t}$. 
 Now since $\int_{-1}^{p}f(t)dt= \int_{-1}^{p}f(t)dt$, $c_1=-c_2$. This completes the proof of the theorem.

\end{proof}


 \end{document}